\documentclass[10pt]{amsart}
\usepackage{graphicx}
\usepackage{amssymb}
\usepackage{amsmath}
\usepackage{amsthm}
\usepackage{epstopdf}
\newtheorem{theorem}{Theorem}
\newtheorem{corollary}[theorem]{Corollary}
\newtheorem{lemma}[theorem]{Lemma}
\newtheorem{definition}{Definition}

\newcommand{\dif}{\mathrm{d}}

\newcommand{\spn}{\mathrm{span}}

\begin{document}

\title{A Natural Connection on $(2,3)$ Sub-Riemannian Manifolds}

\author{Daniel R. Cole}
\address{17 Andiron Ln, Brookhaven, NY 11719}
\email{daniel.cole@aya.yale.edu}
\thanks{The author was supported by NSF grant DMS-0240058.}



\date{October 22, 2008}

\begin{abstract} We build an analogue for the Levi-Civita connection on Riemannian manifolds for sub-Riemannian manfiolds modeled on the Heisenberg group.  We demonstrate some geometric properties of this connection to justify our choice and show that this connection is unique in having these properties. 
\end{abstract}

\maketitle

\section{Introduction}

Sub-Riemannian geometry is an extension of Riemannian geometry in which, instead of defining a smooth, positive-definite inner product on the entire tangent space, we define our inner product only on a sub-bundle of the tangent space.  The result is a wonderfully complex geometry modeled on graded nilpotent Lie groups, instead of just $\mathbb{R}^n$.  

Sub-Riemannian geometry naturally arises in the study of configuration spaces, frame bundles, and principal bundles over Riemannian manifolds.  A beautiful example of a five-dimensional sub-Riemannian manifold is the configuration space of a striped billiard ball rolling on a billiards table: here, the Euclidean metric on the table induces an inner product on the two-dimensional sub-bundle of the tangent space to the configuration space that corresponds to rolling the ball along the surface.  Sub-Riemannian geometry is very strongly related to CR geometry and control theory.  It is also a natural setting for the study of hypoelliptic PDEs.  One recent application of sub-Riemannian manifolds is a model of the first layer of the brain's visual cortex.  Minimal surfaces in this model explain how our brains fill in gaps in our visual fields, and how certain optical illusions arise (see [HP2] for details).

While sub-Riemannian geometry have been studied in various forms for the past century, the differential geometry of sub-Riemannian manifolds is not very well developed at this time. This is due to the relative local complexity of sub-Riemannian manifolds in comparison to their Riemannian cousins.  The metric tangent cone of a bracket-generating sub-Riemannian manifold at any point is either a Carnot group (a graded nilpotent simply connected Lie group) or a quotient of a Carnot group (see [Bel] for a precise statement and proof of this result). The basic structures of Riemannian geometry, including the Levi-Civita connection and normal coordinates, owe their existence to our ability to easily identify the tangent space of a Riemannian manifold at a point (which is isomorphic to $\mathbb{R}^n$) to a local neighborhood of that point (which, in a Riemannian manifold, geometrically approximates $\mathbb{E}^n$). The fact that neighborhoods of points in all but the simplest sub-Riemannian manifolds gemetrically approximate non-abelian Carnot groups or their quotients complicates matters enormously.

This paper makes some progress towards describing the differential geometry of sub-Riemannian manifolds.  Specifically, we consider the simplest non-trivial examples of sub-Riemannian manifolds: those modeled after the three-dimensional Heisenberg group, $\mathbb{H}^1$, which is the simplest of all non-abelian Carnot groups.  Examples of such manifolds, which we call $(2,3)$ sub-Riemannian manifolds (this terminology will be made clear in Section 2) include the roto-translational group, used to model the first layer of the visual cortex (see [HP2]) and the frame bundle on the 2-sphere (see [Mon] for a description of how to induce a sub-Riemannian structure on a frame bundle).  We construct in this paper a natural connection on orientable $(2,3)$ sub-Riemannian manifolds, analogous to the Levi-Civita connection on Riemannian manifolds.  Orientability is, unfortunately, a necessity because, as we will show, this natural connection has a non-zero torsion tensor, and this non-zero torsion tensor induces a global frame on our manifold.

This paper is organized as follows.  In Section 2, we define our basic terms and set notation for the remainder of the paper.  In Section 3, we do a brief study of affine connections on the three-dimensional Heisenberg group.  Our contention is that the most natural choice for a connection on $\mathbb{H}^1$ is the unique affine connection that is compatible with the Lie algebra $\mathfrak{h}$ of $\mathbb{H}^1$: in other words, for all left invariant vector fields $V$ on $\mathbb{H}^1$, we should have \begin{equation} \nabla V = 0 \end{equation} The intuitive idea here is that in the Heisenberg group, or any Carnot group, the left invariant vector fields should be parallel, since they essentially define the structure of the Carnot group, and any tensor which defines the structure of a sub-Riemannian manifold should remain invariant under parallel transport.

In Section 4, we begin our general construction by studying compatible connections on $(2,3)$ sub-Riemannian manifolds.  For a general $(2,3)$ sub-Riemannian manifold, we do not have a Lie algebra to help define the geometric structure of the manifold, but we do have the sub-Riemannian structure, which can be expressed as a tensor.  A connection compatible with the sub-Riemannian structure of the manifold satisfies the relation \begin{equation} \nabla g = 0 \end{equation} where $g$ is the unique co-metric associated to the sub-Riemannian structure.  This is significantly different than any method previously employed to create a connection on sub-Riemannian manifolds or comparable geometric structures (like strictly pseudoconvex pseudohermitian manifolds).  The usual method is to begin by extending the fiber inner product $\left< \cdot, \cdot \right>$ on $\mathcal{H}$ to a full Riemannian metric on $M$.  See, for example, Hladky and Pauls's construction in [HP].  Thus, most constructions begin by making a choice of extension.  With strictly pseudoconvex pseudohermitian manifolds, the usual choice of connection is the Tanaka-Webster connection (see [Tan] and [Web]), but with this connection, too, there is an initial choice: in this case, it is the contact form $\eta$.

In Riemannian geometry, we see that two affine connections $\nabla$ and $\bar{\nabla}$ that are compatible with the Riemannian metric (and thus with the inverse co-metric) are equal if and only if their torsion tensors are equal. Theorem \ref{comptheorem} in Section 4 gives the analogue statement for compatible connections on $(2,3)$ sub-Riemannian manifolds: two affine connections $\nabla$ and $\bar{\nabla}$ that are compatible with the sub-Riemannian structure of $M$ are equal if and only if their torsion tensors are equal \emph{and} their horizontal curvature operators (to be defined in Section 3) are also equal.

In Section 5, we derive the tools we need to find a unique natural compatible connection for orientable $(2,3)$ sub-Riemannian manifolds.  Since compatible connections only differ in their torsion tensor and their horizontal curvature operator (which is also tensorial), we need only be concerned with assigning values to these two tensors smoothly throughout our manifold, and this is essentially a local problem (the only possible global issue is resolved by specifying that the manifold must be orientable).  In the case of Riemannian geometry, this problem is solved by "flattening out" the Riemannian metric around a point $p$ in such a way that first and second order differential operators at $p$ maintain their values, but the geometry in a neighborhood of $p$ is Euclidean.  We then assign the torsion tensor the same value at $p$ as it would have if the surrounding neighborhood were flat, as simulated by the flattening of the Riemannian metric that we did before.  This naturally forces the torsion tensor to be equal to the zero tensor at all points $p$, leading to the Levi-Civita connection being torsion-free.

For orientable $(2,3)$ sub-Riemannian manifolds, the equivalent process is to flatten out a neighborhood of a point $p$ so that the neighborhood looks like the Heisenberg group. In this case, we have an analogous condition on the differential operators, with operators of first, second, and third \emph{weighted} order remaining the same at $p$ (this weighting will be defined in Section 2).  This identification of differential operators at $p$ is only possible when the frame we are using to generate the differential operators has a certain bracket structure at $p$.  Such a frame will be called a Carnot frame, and Theorem \ref{carnotexistence} in Section 5 guarantees the existence of such frames at all points.  The sub-Riemannian structure we get on our neighborhood of $p$ by this process will be called a \emph{flattening} of the sub-Riemannian structure of $M$. Of course, there are many different possible flattenings, but they all yield the same values for the torsion tensor and the horizontal curvature operator.  The easiest way to see this is to show that the flattening generates a Carnot frame (and vice versa), and that all Carnot frames generate the same torsion tensor.  These are the results of Theorem \ref{carnotflatequiv} and Corollary \ref{carnottorsion}, respectively.  These results fix the torsion tensor, which, as noted before, is necessarily non-zero.  The horizontal curvature operator is zero for all flattenings, so this tensor is fixed as well.

Finally, in Section 6, we summarize all of our findings by defining the natural connection on any orientable $(2,3)$ sub-Riemannian manifold with the torsion tensor and horizontal curvature operator as above, and show that it agrees with the geometry of any flattening of the sub-Riemannian structure at any point in $M$.  This main result of the paper is Theorem \ref{maintheoremofpaper}.

The author would like to thank his advisor, S.\@ Pauls, for his continued guidance, R.\@ Hardt for his mentoring, and C.\@ Eyring for her support throughout the writing of this paper.  He would also like to thank C.\@ Douglas, B.\@ Paier, and R.\@ Dunning, along with R.\@ Hardt, for their attendance and helpful comments throughout a seminar that the author gave on this topic.

\section{Basic Definitions and Notation}

This section gives the basic definitions used in sub-Riemannian geometry.  See [Mon] and [Bel] for more detailed introductions to the subject.

Let $M$ is an orientable smooth manifold of dimension $n$.

\begin{definition} A sub-Riemannian structure on $M$ is an ordered pair $\left(\mathcal{H}, \left< \cdot, \cdot \right> \right)$, where $\mathcal{H}$ is a sub-bundle of the tangent bundle of dimension $m \leq n$ called the \emph{horizontal sub-bundle}, and $\left< \cdot, \cdot \right>$ is a fiber inner product on $\mathcal{H}$.
\end{definition}

Alternatively, as outlined in Montgomery, we can define a sub-Riemannian structure on $M$ using a co-metric $g$. Specifically, $g$ is the unique co-metric such that,  for all $p \in M$ and $\xi_1, \xi_2 \in T^\ast_p M$, \begin{equation} g(\xi_1, \cdot), g(\xi_2, \cdot) \in \mathcal{H}_p \end{equation} where the usual canonical identification is being made between $T_p M$ and the space of linear functionals on $T^\ast_p M$; and \begin{equation} g(\xi_1, \xi_2) = \left< g(\xi_1, \cdot), g(\xi_2, \cdot) \right> \end{equation} If, on an open neighorbood $U$ of $M$, $\mathcal{H}$ has an ordered orthonormal frame $\{X_1, \ldots, X_m\}$, then on $U$ we have that \begin{equation} g = X_1 \otimes X_1 + \cdots + X_m \otimes X_m \end{equation} This characterization of $g$ will be particularly useful for us in the following sections.

For all integers $k \geq 1$, we deine the sub-sheaf $\mathcal{H}^k$ by the following recursive definition: \begin{eqnarray} \mathcal{H}^1 &=& \mathcal{H} \\ \mathcal{H}^k &=& \left\{ f_1 V_1 + f_2 [V_2, W] \,|\, f_1, f_2 \in C^\infty(M),\, V_1, V_2 \in \mathcal{H}^{k-1}, \, W \in \mathcal{H} \right\} \quad \text{for $k \geq 2$}\end{eqnarray} Note that $\mathcal{H}^k$ is not necessarily a sub-bundle of $TM$ for $k \geq 2$ because the dimension of $\mathcal{H}^k$ may not be constant over all of $M$.

\begin{definition} The sub-bundle $\mathcal{H}$ is \emph{bracket generating} over $M$ if for some finite $r$, $\mathcal{H}^r = TM$. \end{definition}

\noindent Define $m_k(p)$ to be the dimension of $\mathcal{H}^k$ at $p$.  Define $r(p)$ to be the least integer $k$ such that $\mathcal{H}^k_p = T_p M$.

\begin{definition} Let $\mathcal{H}$ be a bracket generating sub-bundle. The \emph{growth vector} of $\mathcal{H}$ at $p$ is $r(p)$-tuple \begin{equation} \left( m_1(p), m_2(p), \ldots, m_{r(p)}(p) \right) \end{equation} \end{definition}

\begin{definition} A \emph{$(2,3)$ sub-Riemannian manifold} is a smooth manifold $M$ of dimension 3 coupled with a bracket generating sub-Riemannian structure $\left(\mathcal{H}, \left< \cdot, \cdot \right> \right)$ with $\mathcal{H}$ of dimension 2 such that $\mathcal{H}$ has growth vector $(2,3)$ at all points in $M$. \end{definition}

\noindent For the remainder of this paper, we will assume that $M$ is an orientable $(2,3)$ sub-Riemannian manifold with sub-Riemannian structure $\left(\mathcal{H}, \left< \cdot, \cdot \right> \right)$ equivalent to a co-metric $g$.

Let $U$ be an open neighborhood of $M$.  Let $\{X_1, X_2 \}$ is an ordered orthonormal frame for $\mathcal{H}$ on $U$ such that, defining $X_3 = [X_1,X_2]$, $\{X_1, X_2, X_3 \}$ is an oriented frame for $TU$ with orientation matching that of $M$.  Let $\{ \xi_1, \xi_2, \xi_3 \}$ be the dual frame to $\{X_1, X_2, X_3 \}$.

Let $p \in U$.  In Section 5, we will need to work with differential operators at $p$ of the form \begin{equation} X^\alpha_p \mathrel{\mathop:}= \left.X_{\alpha_1} X_{\alpha_2} \cdots X_{\alpha_N}\right|_p, \quad \alpha_i \in \{1, 2, 3\} \end{equation} where here we are using a multi-index notation with $\alpha = (\alpha_1, \alpha_2, \ldots, \alpha_N)$.  We define the \emph{order} $|\alpha|$ of $\alpha$ to be the length of the multi-index, and the \emph{weighted order} $|\alpha|_w$ of $\alpha$ to be the number of $\alpha_i$ equal to 1 or 2 plus twice the number of $\alpha_i$ equal to 3. If $\alpha$ has order $N$ or weighted order $M$, we say that $X^\alpha$ has order $N$ or weighted order $M$ at $p$, respectively.  

Admittedly, this is a ``quick and dirty'' way to define the weighted order of a differential operator on a sub-Riemannian manifold, ignoring issues of whether weighted order is well-defined, but since we will only be using this definition for bookkeeping and to shorten some definition statements, this definition is sufficiently rigorous for our purposes.  For a far more rigorous notion of weighted order of a differential operator on a sub-Riemannian manifold, see [Bel].

\section{Compatible Connections on the Heisenberg Group}

Recall the three-dimensional Heisenberg group $\mathbb{H}^1$.  This graded nilpotent Lie group has Lie algebra \begin{equation} \mathfrak{h} = \spn\{X_1, X_2, X_3\} \end{equation} where $X_1$, $X_2$, and $X_3$ satisfy the Heisenberg bracket relations \begin{equation} [X_1, X_2] = X_3 \qquad  [X_3, X_1] = 0 \qquad [X_2, X_3] = 0 \end{equation}  The group operation on $\mathbb{H}^1$ is given, as usual, by the Campbell-Baker-Hausdorff formula: for all $V_1, V_2 \in \mathfrak{h}$, \begin{equation} \exp(V_1) \circ \exp(V_2) = \exp\left(V_1 + V_2 + \frac {1} {2} [V_1,V_2]\right) \end{equation}  The frame $\{X_1, X_2, X_3\}$ is left invariant under this group action, as is its dual frame $\{\xi^1, \xi^2, \xi^3\}$.  Define a sub-Riemannian structure $\left(\mathcal{H}, \left< \cdot, \cdot \right> \right)$ on $\mathbb{H}^1$ such that \begin{equation} \mathcal{H} = \spn \{X_1, X_2 \} \end{equation} and $\{X_1, X_2\}$ is an oriented orthonormal frame for $\mathcal{H}$.  This sub-Riemannian structure is clearly left invariant under the above group action, as is the associated co-metric \begin{equation} g = X_1 \otimes X_1 + X_2 \otimes X_2 \end{equation}

\begin{definition}An affine connection $\nabla$ is \emph{compatible} with the Lie algebra $\mathfrak{h}$ if for all $V \in \mathfrak{h}$, \begin{equation} \nabla V = 0 \end{equation} \end{definition}

\begin{definition} An affine connection $\nabla$ is \emph{compatible} with $\left( \mathcal{H}, \left< \cdot, \cdot \right> \right)$, and thus its associated co-metric $g$, if \begin{equation} \nabla g = 0 \end{equation} 
\end{definition}

\begin{lemma} There exists a unique affine connection $\nabla$ compatible with $\mathfrak{h}$.  If $\nabla$ is compatible with $\mathfrak{h}$, then $\nabla$ is compatible with g, and thus with $\left(\mathcal{H}, \left< \cdot, \cdot \right> \right)$. \end{lemma}

\begin{proof} The existence and uniqueness of $\nabla$ is clear since any affine connection is determined by its action on a frame for $\mathbb{H}^1$, and since $\nabla$ is compatible with $\mathfrak{h}$, we must have $\nabla X_1 = 0$, $\nabla X_2 = 0$, and $\nabla X_ 3 = 0$.  If $\nabla$ is compatible with $\mathfrak{h}$, then \begin{eqnarray} \nabla g &=& \nabla X_1 \otimes X_1 + X_1 \otimes \nabla X_1 + \nabla X_2 \otimes X_2 + X_2 \otimes \nabla X_2 \\ &=& 0 \otimes X_1 + X_1 \otimes 0 + 0 \otimes X_2 + X_2 \otimes 0 \\ &=& 0 \end{eqnarray} Thus $\nabla$ is compatible with $g$. \end{proof}

For any sub-Riemannian manifold $M$, define the torsion tensor $T$ of $\nabla$ by the usual formula \begin{equation} T(V_1, V_2) = \nabla_{V_1} V_2 - \nabla_{V_2} V_1 - [V_1, V_2] \end{equation} for all $V_1, V_2 \in \mathcal{X}(M)$.  Define the curvature operator $R$ as usual by \begin{equation} R(V_1, V_2)V_3 = \nabla_{V_1} \nabla_{V_2} V_3 - \nabla_{V_2} \nabla_{V_1} V_3 - \nabla_{[V_1, V_2]} V_3 \end{equation} for all $V_1, V_2, V_3 \in \mathcal{X}(M)$.  We make the following definition.

\begin{definition} Let $\{ X_1, X_2 \}$ be an oriented orthonormal frame for $\mathcal{H}$ on $U \in M$, and let $X_3 = [X_1, X_2]$.  We define the \emph{horizontal curvature operator} of $\nabla$ to be \begin{equation} R_\mathcal{H} V = R(X_1, X_2)V = \nabla_{X_1} \nabla_{X_2} V - \nabla_{X_2} \nabla_{X_1} V - \nabla_{X_3} V \end{equation} for all $V \in \mathcal{X}(U)$.
\end{definition}

We note that we can write any other oriented orthonormal frame $\{X_1',X_2'\}$ for $\mathcal{H}$ as \begin{equation} \label{changeofframe} X_1' = \cos \theta \, X_1 + \sin \theta \, X_2 \qquad X_2' = -\sin \theta \, X_1 + \cos \theta \, X_2 \end{equation} where $\theta : U \rightarrow \mathbb{R}$ is smooth.

\begin{lemma}
The horizontal curvature operator $R_{\mathcal{H}}$ is independent of the choice of oriented orthonormal frame $\{ X_1, X_2 \}$ for $\mathcal{H}$, and thus is well defined.
\end{lemma}

\begin{proof}
Suppose $\{ X'_1, X'_2 \}$ be an oriented orthonormal frame for $\mathcal{H}$ on $U$.  Applying equation (\ref{changeofframe}) and the property of $R$ being an alternating tensor, we see that \begin{eqnarray} R(X'_1, X'_2) & = & R(\cos \theta \, X_1 + \sin \theta \, X_2, -\sin \theta \, X_1 + \cos \theta \, X_2) \\ & = & -\sin \theta \cos \theta R(X_1, X_1) + (\cos^2 \theta + \sin^2 \theta) R(X_1, X_2) + \sin \theta \cos \theta R(X_2, X_2) \\ & = & R(X_1, X_2) \end{eqnarray} so $R_\mathcal{H}$ is independent of the frame used to define it.
\end{proof}

Finally, we prove the following lemma on $\mathbb{H}^1$, to be used later.

\begin{lemma} \label{torsioncurvature} If $\nabla$ is compatible with $\mathfrak{h}$, then 
\begin{enumerate}
\item[(a)] $T(X_1, X_2) = -X_3$
\item[(b)] $T(X_2, X_3) = 0$
\item[(c)] $T(X_3, X_1) = 0$
\item[(d)] $R_\mathcal{H} V = 0$ for all $V \in \mathcal{X}\left(\mathbb{H}^1\right)$
\end{enumerate}
\end{lemma}

\begin{proof} The reader can quickly verfiy this lemma using the definitions of the torsion tensor and the horizontal curvature operator. \end{proof}

\section{Compatible Connections on $(2,3)$ Sub-Riemannian Manifolds}

Let $M$ be an orientable $(2,3)$ sub-Riemannian manifold with sub-Riemannian structure $\left( \mathcal{H}, \left< \cdot, \cdot \right> \right)$.  We now discuss the conditions that an affine connection $\nabla$ must meet in order to be compatible with our sub-Riemannian structure.  Again, let $g$ be the unique co-metric associated to $\left( \mathcal{H}, \left< \cdot, \cdot \right> \right)$.  Then if $\{ X_1, X_2 \}$ is an oriented orthonormal frame for $\mathcal{H}$, then \begin{equation} g = X_1 \otimes  X_1+ X_2 \otimes X_2 \end{equation}

\begin{lemma} \label{compatcond} Let $\{X_1, X_2 \}$ be an ordered orthonormal frame for $\mathcal{H}$ on $U$ and let $X_3 = [X_1, X_2]$.  Let $\{\xi^1, \xi^2, \xi^3\}$ be the dual frame to $\{X_1, X_2, X_3 \}$  If $\nabla$ is compatible with $g$, then for all $V \in \mathcal{X}(U)$
\begin{enumerate}
\item[(a)] $\xi^1\left( \nabla_V X_1 \right) = 0 = \xi^2\left( \nabla_V X_2 \right)$
\item[(b)] $\xi^2\left( \nabla_V X_1 \right) = -\xi^1\left( \nabla_V X_2 \right)$
\item[(c)] $\xi^3\left( \nabla_V X_1 \right) = 0 = \xi^3\left( \nabla_V X_2 \right)$
\end{enumerate}
\end{lemma}

\begin{proof}
Applying $\nabla$ to $g$, we get \begin{equation} \label{nablag} 0 = \nabla g = \nabla X_1 \otimes X_1 + X_1 \otimes \nabla X_1 + \nabla X_2 \otimes X_2 + X_2 \otimes \nabla X_2 \end{equation}  Let $V \in \mathcal{X}(U)$.  For all $i, j \in \{1,2,3\}$, we get from equation (\ref{nablag}) that \begin{equation} 0 = (\nabla_V g )(\xi^i, \xi^j) = \xi^i (\nabla_V X_1) \xi^j (X_1) + \xi^i (X_1) \xi^j (\nabla_V X_1) + \xi^i (\nabla_V X_2) \xi^j (X_2) + \xi^i (X_2) \xi^j (\nabla_V X_2) \end{equation} In particular, taking $(i, j) = (1,1)$ gives us \begin{equation} 0 = \xi^1 (\nabla_V X_1) \cdot 1 + 1 \cdot \xi^1 (\nabla_V X_1) + \xi^1 (\nabla_V X_2) \cdot 0 + 0 \cdot \xi^1 (\nabla_V X_2) = 2 \xi^i (\nabla_V X_1) \end{equation} and taking $(i, j) = (2,2)$ gives us \begin{equation} 0 = \xi^2 (\nabla_V X_1) \cdot 0 + 0 \cdot \xi^2 (\nabla_V X_1) + \xi^2 (\nabla_V X_2) \cdot 1 + 1 \cdot \xi^2 (\nabla_V X_2) = 2 \xi^2 (\nabla_V X_2) \end{equation} so part (a) of the lemma is true.  Likewise, taking $(i, j) = (1, 3)$ and $(i, j) = (2,3)$ gives us, respectively, \begin{equation} 0 = \xi^1 (\nabla_V X_1) \cdot 0 + 1 \cdot \xi^3 (\nabla_V X_1) + \xi^1 (\nabla_V X_2) \cdot 0 + 0 \cdot \xi^3 (\nabla_V X_2) = \xi^3 (\nabla_V X_1) \end{equation} and \begin{equation} 0 = \xi^2 (\nabla_V X_1) \cdot 0 + 0 \cdot \xi^3 (\nabla_V X_1) + \xi^2 (\nabla_V X_2) \cdot 0 + 1 \cdot \xi^3 (\nabla_V X_2) = \xi^3 (\nabla_V X_2) \end{equation} proving part (c) of the lemma.  Finally, taking $(i, j) = (1, 2)$, we get that \begin{equation} 0 = \xi^1 (\nabla_V X_1) \cdot 0 + 1 \cdot \xi^2 (\nabla_V X_1) + \xi^1 (\nabla_V X_2) \cdot 1 + 0 \cdot \xi^2 (\nabla_V X_2) = \xi^2 (\nabla_V X_1) + \xi^1 (\nabla_V X_2) \end{equation} so part (b) of the lemma is true as well.
\end{proof}

\begin{corollary}  \label{expansioncor} There exist smooth functions $f_i : U \rightarrow \mathbb{R}$, $i \in \{1, 2, 3\}$ such that  \begin{equation} \nabla_{X_i} X_1 = f_i X_2 \qquad \text{and} \qquad \nabla_{X_i} X_2 = -f_i X_1 \end{equation} for all $i \in \{1, 2, 3\}$.
\end{corollary}

\begin{proof}
For $i \in \{1, 2, 3 \}$, set \begin{equation} f_i = \xi^2\left( \nabla_{X_i} X_1 \right) \end{equation}  The dual frame $\{ \xi^1, \xi^2, \xi^3 \}$ is orthonormal under the co-metric $g$, thus for all $i, j \in \{1, 2, 3\}$, we have that \begin{equation} \label{nablaexpansion} \nabla_{X_i} X_j = \left( \xi^1 (\nabla_{X_i} X_j) \right) X_1 + \left( \xi^2 (\nabla_{X_i} X_j) \right) X_2 + \left( \xi^3 (\nabla_{X_i} X_j) \right) X_3 \end{equation}  Applying Lemma \ref{compatcond} to equation (\ref{nablaexpansion}) for $i \in \{1, 2, 3\}$ and $j \in \{1, 2 \}$ proves the corollary.
\end{proof}

For compatible connections $\nabla$, we have the following lemma.

\begin{lemma} \label{torsioncond}
Let $\{X_1, X_2, X_3\}$ and $\{ \xi^1, \xi^2, \xi^3 \}$ be as above, and let $\nabla$ be an affine connection compatible with the co-metric $g$.  Then
\begin{enumerate}
\item[(a)] $\xi^1(T(X_1, X_2)) = \xi^1\left( \nabla_{X_1} X_2 \right) = -\xi^2\left( \nabla_{X_1} X_1 \right)$
\item[(b)] $\xi^2(T(X_1, X_2)) = \xi^1\left( \nabla_{X_2} X_2 \right) = -\xi^2\left( \nabla_{X_2} X_1 \right)$
\item[(c)] $\xi^3(T(X_1, X_2)) = -1$
\end{enumerate}
In particular, if $\nabla$ is compatible with $g$, then $\nabla$ cannot be torsion-free.
\end{lemma}

\begin{proof}
Computing $T(X_1, X_2)$ and applying Corollary $\ref{expansioncor}$, we get \begin{eqnarray} T(X_1, X_2) & = & \nabla_{X_1} X_2 - \nabla_{X_2} X_1 - [X_1, X_2] \\ & = & -f_1 X_1 - f_2 X_2 - X_3 \end{eqnarray} which implies parts (a)--(c) of the lemma.
\end{proof}

We now begin investigating the horizontal curvature operator of a compatible connection.

\begin{lemma} \label{curvaturecond}
Let $\{X_1, X_2, X_3\}$ and $\{ \xi^1, \xi^2, \xi^3 \}$ be as above, and let $\nabla$ be an affine connection compatible with the co-metric $g$.  Then
\begin{enumerate}
\item[(a)] $\xi^1(R_\mathcal{H} X_1) = 0 = \xi^2(R_\mathcal{H} X_2)$
\item[(b)] $\xi^2(R_\mathcal{H} X_1) = -\xi^1(R_\mathcal{H} X_2)$
\item[(c)] $\xi^3(R_\mathcal{H} X_1) = \xi^3(R_\mathcal{H} X_2)$
\end{enumerate}
\end{lemma}

\begin{proof}
Computing $R_\mathcal{H} X_1$, we get
\begin{eqnarray} R_\mathcal{H} X_1 & = & \nabla_{X_1} \nabla_{X_2} X_1 - \nabla_{X_2} \nabla_{X_1} X_1 - \nabla_{X_3} X_1 \\
& = & \nabla_{X_1} (f_2 X_2) - \nabla_{X_2} (f_1 X_2) - f_3 X_2 \\
& = & (X_1 f_2) X_2 - f_2 f_1 X_1 - (X_2 f_1) X_2 + f_1 f_2 X_1 - f_3 X_2 \\
\label{RH1expansion} & = & (X_1 f_2 - X_2 f_1 - f_3) X_2 
\end{eqnarray} Likewise, computing $R_\mathcal{H} X_2$, we see that
\begin{eqnarray} R_\mathcal{H} X_2 & = & \nabla_{X_1} \nabla_{X_2} X_2 - \nabla_{X_2} \nabla_{X_1} X_2 - \nabla_{X_3} X_2 \\
& = & \nabla_{X_1} (-f_2 X_1) - \nabla_{X_2} (-f_1 X_1) + f_3 X_1 \\
& = & -(X_1 f_2) X_1 - f_2 f_1 X_2 + (X_2 f_1) X_1 + f_1 f_2 X_2 + f_3 X_1 \\
\label{RH2expansion} & = & -(X_1 f_2 - X_2 f_1 - f_3) X_1 
\end{eqnarray} Parts (a)--(c) then follow from equations (\ref{RH1expansion}) and (\ref{RH2expansion}).
\end{proof}

We now come to the main theorem of this section.

\begin{theorem} \label{comptheorem} Suppose $\nabla$ and $\bar{\nabla}$ are two affine connections compatible with the co-metric $g$.  Let $T$ and $R_\mathcal{H}$ be the torsion tensor and the horizontal curvature operator of $\nabla$, and let $\bar{T}$ and $\bar{R}_\mathcal{H}$ be the torsion tensor and the horizontal curvature operator of $\bar{\nabla}$.  Then $\nabla = \bar{\nabla}$ if and only if $T = \bar{T}$ and $R_\mathcal{H} = \bar{R}_\mathcal{H}$.
\end{theorem}

\begin{proof} Let $\{X_1, X_2, X_3\}$ and $\{ \xi^1, \xi^2, \xi^3 \}$ be as above.  Suppose $T = \bar{T}$ and $R_\mathcal{H} = \bar{R}_\mathcal{H}$.  Then by Lemma \ref{compatcond} and Lemma \ref{torsioncond}, we have that
\begin{eqnarray} && \bar{\nabla}_{X_1} X_1 = -\xi^1(\bar{T}(X_1, X_2)) X_2 = -\xi^1(T(X_1, X_2)) X_2 = \nabla_{X_1} X_1 \\
&& \bar{\nabla}_{X_1} X_2 = \xi^1(\bar{T}(X_1, X_2)) X_1 = \xi^1(T(X_1, X_2)) X_1 = \nabla_{X_1} X_2 \\
&& \bar{\nabla}_{X_2} X_1 = -\xi^2(\bar{T}(X_1, X_2)) X_2 = -\xi^2(T(X_1, X_2)) X_2 = \nabla_{X_2} X_1 \\
&& \bar{\nabla}_{X_2} X_2 = \xi^2(\bar{T}(X_1, X_2)) X_1 = \xi^2(T(X_1, X_2)) X_1 = \nabla_{X_2} X_2
\end{eqnarray} By Corollary \ref{expansioncor} and Lemma \ref{curvaturecond}, we also see that 
\begin{eqnarray} \bar{\nabla}_{X_3} X_1 & = & \bar{\nabla}_{X_1} \bar{\nabla}_{X_2} X_1 - \bar{\nabla}_{X_2} \bar{\nabla}_{X_1} X_1 - \bar{R}_\mathcal{H}X_1 \\
& = & [- X_1 [\xi^1 (\bar{T}(X_1, X_2))] + X_2 [\xi^2 (\bar{T}(X_1, X_2))]] X_2 - \bar{R}_\mathcal{H} X_1 \\
& = & [- X_1 [\xi^1 (T(X_1, X_2))] + X_2 [\xi^2 (T(X_1, X_2))]] X_2 - R_\mathcal{H} X_1 \\
& = & \nabla_{X_3} X_1 \end{eqnarray} and similarly
\begin{eqnarray} \bar{\nabla}_{X_3} X_2 & = & \bar{\nabla}_{X_1} \bar{\nabla}_{X_2} X_2 - \bar{\nabla}_{X_2} \bar{\nabla}_{X_1} X_2 - \bar{R}_\mathcal{H}X_2 \\
& = & [X_1 [\xi^1 (\bar{T}(X_1, X_2))] - X_2 [\xi^2 (\bar{T}(X_1, X_2))]] X_1 - \bar{R}_\mathcal{H} X_2 \\
& = & [X_1 [\xi^1 (T(X_1, X_2))] - X_2 [\xi^2 (T(X_1, X_2))]] X_1 - R_\mathcal{H} X_2 \\
& = & \nabla_{X_3} X_2 \end{eqnarray} We then apply the definition of the torsion tensor to get
\begin{eqnarray} \bar{\nabla}_{X_1} X_3 & = & \bar{T}(X_1, X_3) + \bar{\nabla}_{X_3} X_1 + [X_1, X_3] \\
& = & T(X_1, X_3) + \nabla_{X_3} X_1 + [X_1, X_3] \\
& = & \nabla_{X_1} X_3 \end{eqnarray} and, in the same way,
\begin{eqnarray} \bar{\nabla}_{X_2} X_3 & = & \bar{T}(X_2, X_3) + \bar{\nabla}_{X_3} X_2 + [X_2, X_3] \\
& = & T(X_2, X_3) + \nabla_{X_3} X_2 + [X_2, X_3] \\
& = & \nabla_{X_2} X_3 \end{eqnarray}  Thus we have that, for any $V \in \mathcal{X}(U)$, \begin{equation} \label{nablabarequalsnabla} \bar{\nabla}_{X_1} V = \nabla_{X_1} V \qquad \text{and} \qquad \bar{\nabla}_{X_2} V = \nabla_{X_2} V \end{equation} Finally, we use equation (\ref{nablabarequalsnabla}) and the formula for the horizontal curvature operator to get
\begin{eqnarray} \bar{\nabla}_{X_3} X_3 & = & \bar{\nabla}_{X_1} \bar{\nabla}_{X_2} X_3 - \bar{\nabla}_{X_2} \bar{\nabla}_{X_1} X_3 - \bar{R}_\mathcal{H}X_3 \\
& = & \bar{\nabla}_{X_1} \nabla_{X_2} X_3 - \bar{\nabla}_{X_2} \nabla_{X_1} X_3 - R_\mathcal{H}X_3 \\
& = & \nabla_{X_1} \nabla_{X_2} X_3 - \nabla_{X_2} \nabla_{X_1} X_3 - R_\mathcal{H}X_3 \\
& = & \nabla_{X_3} X_3 \end{eqnarray}  Thus, on any coordinate patch $U$, and hence on all of $M$, $\bar{\nabla} = \nabla$.

If $\nabla = \bar{\nabla}$, then $T = \bar{T}$ and $R_\mathcal{H} = \bar{R}_\mathcal{H}$ as a trivial result of the definitions of the torsion tensor and the horizontal curvature operator. \end{proof}

\section{Flattening the Sub-Riemannian Structure and Carnot Frames}

Theorem \ref{comptheorem} tells us that affine connections that are compatible with the sub-Riemannian structure of $M$ are determined by their torsion tensor and their horizontal curvature operator.  We now begin the process of finding natural values for the torsion tensor and the horizontal curvature operator.  We begin with a definition.

\begin{definition} Let $\hat{U}$ be an open neighborhood of $p$, and let $\{\hat{X}_1, \hat{X}_2, \hat{X}_3 \}$ be a frame for $T\hat{U}$ such that, for all $q \in \hat{U}$, \begin{equation} [\hat{X}_1, \hat{X}_2]_q = (\hat{X}_3)_q \qquad  [\hat{X}_3, \hat{X}_1]_q = 0 \qquad [\hat{X}_2, \hat{X}_3]_q = 0 \end{equation} We say that \begin{equation} \hat{g} = \hat{X}_1 \otimes \hat{X}_1 + \hat{X}_2 \otimes \hat{X}_2 \end{equation} is a \emph{flattening} of $g$ at $p$ with frame $\{\hat{X}_1, \hat{X}_2, \hat{X}_3 \}$ if there exists an ordered orthonormal frame $\{X_1, X_2 \}$ for $\mathcal{H}$ on $\hat{U}$ such that for all first, second, and third weighted order multi-indices $\alpha$, \begin{equation} \hat{X}^\alpha_p = X^\alpha_p \end{equation} as smooth differential operators at p. \end{definition}

Thus a flattening $\hat{g}$ at $p$ with frame $\{\hat{X}_1, \hat{X}_2, \hat{X}_3 \}$ is isomorphic to an open neighborhood of the Heisenberg group $\mathbb{H}^1$, and it well approximates the horizontal sub-bundle $\mathcal{H}$ of $M$ in the neighborhood of $p$.  We measure this approximation by studying the first, second, third weighted order differential operators at $p$ generated by both the frame $\{X_1, X_2, X_3\}$ and the frame $\{\hat{X}_1, \hat{X}_2, \hat{X}_3 \}$.  In general, we cannot find a frame $\{\hat{X}_1, \hat{X}_2, \hat{X}_3 \}$ that approximates $\{X_1, X_2, X_3\}$ at $p$ to fourth or higher weighted order.

To understand the structure of a flattening (and, in particular, the direction of $\hat{X}_3$ at $p$), we now consider a different way to recover the structure of the Heisenberg group into our sub-Riemannian structure.  Again, suppose $\{X_1, X_2 \}$ is an ordered orthonormal frame for $\mathcal{H}$ on $U$ such that, defining $X_3 = [X_1,X_2]$, $\{X_1, X_2, X_3 \}$ is an oriented frame for $TU$ with orientation matching that of $M$.  Let $\{ \xi_1, \xi_2, \xi_3 \}$ be the dual frame to $\{X_1, X_2, X_3 \}$.  As noted before, we can write any other such frame $\{X_1',X_2'\}$ as \begin{equation} X_1' = \cos \theta \, X_1 + \sin \theta \, X_2 \qquad X_2' = -\sin \theta \, X_1 + \cos \theta \, X_2 \end{equation} where $\theta : U \rightarrow \mathbb{R}$ is smooth.  It then follows from basic calculations that \begin{equation} \label{x3value} X_3' = [X_1', X_2'] = -(X_1 \theta) \, X_1 -(X_2 \theta) \, X_2 + X_3 \end{equation} and \begin{equation} [X_1', X_3'] = \cos \theta \, V_1 + \sin \theta \, V_2 \qquad [X_2', X_3' ] = -\sin \theta \, V_1 + \cos \theta \, V_2 \end{equation} where \begin{eqnarray} \label{V1} V_1 & = & [X_1, X_3] - (X_2 \theta) \, X_3 - (X_1 X_1 \theta) \, X_1 + \left( (X_1 \theta)^2 + (X_2 \theta)^2 - 2 X_1 X_2 \theta + X_2 X_1 \theta \right) \, X_2 \\ \label{V2} V_2 & = & [X_2, X_3] + (X_1 \theta) \, X_3 - (X_2 X_2 \theta) \, X_2 - \left( (X_1 \theta)^2 + (X_2 \theta)^2 - X_1 X_2\theta + 2 X_2 X_1 \theta \right) \, X_1 \end{eqnarray}  With the Heisenberg bracket relations in mind, we make the following definition:

\begin{definition} A \emph{Carnot frame} for $\mathcal{H}$ at $p \in U$ is an ordered orthonormal frame $\{X_1, X_2\}$ for $\mathcal{H}$ such that \begin{equation} [X_1, X_3]_p = 0 \quad \text{and} \quad [X_2, X_3]_p  = 0 \end{equation} \end{definition}

\begin{theorem} \label{carnotexistence} Let $M$, $\left(\mathcal{H}, \left< \cdot, \cdot \right> \right)$, $U$ be as above, and let $p \in U$.  There exists a Carnot frame for $\mathcal{H}$ at $p$. \end{theorem}

\begin{proof} Let $\{X_1, X_2, X_3 \}$, $\{ \xi_1, \xi_2, \xi_3 \}$, and $\{X_1',X_2', X_3' \}$ be defined as above.  We can choose $\theta$ such that \begin{eqnarray} \label{thetaderivative1} (X_1)_p \theta & = & -(\xi_3)_p ([X_2, X_3]) \\  (X_2)_p \theta & = & (\xi_3)_p ([X_1, X_3]) \\ (X_1 X_1)_p \theta & = & (\xi_1)_p ([X_1, X_3]) \\ (X_1 X_2)_p \theta & = & \frac {\left( (\xi_3)_p ([X_1, X_3]) \right)^2 + \left( (\xi_3)_p ([X_2, X_3])  \right)^2 + (\xi_1)_p ([X_2, X_3]) + 2 (\xi_2)_p ([X_1, X_3])} {3}  \\  (X_2 X_1)_p \theta & = & \frac {- \left( (\xi_3)_p ([X_1, X_3]) \right)^2 - \left( (\xi_3)_p ([X_2, X_3])  \right)^2 + (\xi_2)_p ([X_1, X_3]) + 2 (\xi_1)_p ([X_2, X_3])} {3} \\ \label{thetaderivative6} (X_2 X_2)_p \theta & = & (\xi_2)_p ([X_2, X_3])  \end{eqnarray}  Substituting these values into equations (\ref{V1}) and (\ref{V2}) gives us $(V_1)_p = 0$ and $(V_2)_p = 0$, which in turn tells us that \begin{equation} [X_1', X_3']_p = 0 \quad \text{and} \quad [X_2', X_3']_P = 0 \end{equation}  Thus $\{X_1',X_2' \}$ is a Carnot frame for $\mathcal{H}$ at $p$.
\end{proof}

An important implication of $\{X_1, X_2 \}$ being a Carnot frame is that the direction of $X_3 = [X_1, X_2]$ is fixed at $p$, as shown in the following corollary.

\begin{corollary} \label{carnottorsion} Suppose $\{X_1, X_2 \}$ and $\{X_1',X_2' \}$ are as above and both are Carnot frames for $\mathcal{H}$ at $p$.  Then all first and second order horizontal derivatives of $\theta$ at $p$ equal $0$, and $(X_3)_p = (X_3')_p$. \end{corollary} 

\begin{proof} Since $\{X_1, X_2 \}$ is a Carnot frame for $\mathcal{H}$ at $p$, $[X_1, X_3]_p = 0$ and $[X_2, X_3]_p = 0$.  The corollary then follows from equations (\ref{thetaderivative1}) through (\ref{thetaderivative6}).\end{proof} 

The significance of having a Carnot frame for $\mathcal{H}$ at $p$ becomes apparent with the following theorem.  We need the next two lemmata to prove this theorem.

\begin{lemma} \label{2ordlem} Let $\{X_1, X_2\}$ be an oriented orthonormal frame for $\mathbb{H}$ on $U$, and let $X_3 = [X_1, X_2]$.  There exist coordinate functions $(x^1, x^2, x^3)$ on an open neighborhood $U' \subseteq U$ of $p$ such that
\begin{enumerate}
\item[(a)] $x^i (p) = 0$ for all $i \in {1,2,3}$;
\item[(b)] $(X_i  x^j)(p) = \delta^i_j$ for all $i, j \in {1,2,3}$;
\item[(c)] $(X_i X_j x^k)(p) = 0$ for all $i, j, k \in {1, 2, 3}$ except $(i, j, k) \in \{(1, 2, 3), (2, 1, 3)\}$; and
\item[(d)] $(X_1 X_2 x^3)(p) = \frac {1} {2}$ and $(X_2 X_1 x^3)(p) = -\frac {1} {2}$
\end{enumerate}
\end{lemma}

\begin{proof}
Let Let $(x^1, x^2, x^3)$ be the coordinate functions associated to a chart $(U, \phi)$.  After post-composing our coordinate functions with a translation in $\mathbb{R}^3$, part (a) is evident, and assuming part (a), we can get part (b) by post-composing with a linear transformation of $\mathbb{R}^3$.  Without loss of generality then, we may assume parts (a) and (b) are true for $(x^1, x^2, x^3)$ in order to prove parts (c) and (d).  Define the following functions on $U$: \begin{eqnarray} \bar{x}^1 & = & x^1 - \frac {1} {2} \left[ (X_1 X_1)_p x^1 \right] (x^1)^2 - \left[ (X_1 X_2)_p x^1 \right] x^1 x^2 - \frac {1} {2} \left[ (X_2 X_2)_p x^1 \right] (x^2)^2 \\ 
\label{x12ndline} & = & x^1 - \frac {1} {2} \left[ (X_1 X_1)_p x^1 \right] (x^1)^2 - \left[ (X_2 X_1)_p x^1 \right] x^1 x^2 - \frac {1} {2} \left[ (X_2 X_2)_p x^1 \right] (x^2)^2 \\
\bar{x}^2 & = & x^2 - \frac {1} {2} \left[ (X_1 X_1)_p x^2 \right] (x^1)^2 - \left[ (X_1 X_2)_p x^2 \right] x^1 x^2 - \frac {1} {2} \left[ (X_2 X_2)_p x^2 \right] (x^2)^2 \\ 
\label{x22ndline} & = & x^2 - \frac {1} {2} \left[ (X_1 X_1)_p x^2 \right] (x^1)^2 - \left[ (X_2 X_1)_p x^2 \right] x^1 x^2 - \frac {1} {2} \left[ (X_2 X_2)_p x^2 \right] (x^2)^2 \\
\bar{x}^3 & = & x^3 - \frac {1} {2} \left[ (X_1 X_1)_p x^3 \right] (x^1)^2 - \left( \left[ (X_1 X_2)_p x^3 \right]  - \frac {1} {2} \right) x^1 x^2 - \frac {1} {2} \left[ (X_2 X_2)_p x^3 \right] (x^2)^2 \\ 
\label{x32ndline} & = & x^3 - \frac {1} {2} \left[ (X_1 X_1)_p x^3 \right] (x^1)^2 - \left( \left[ (X_2 X_1)_p x^3 \right]  + \frac {1} {2} \right) x^1 x^2 - \frac {1} {2} \left[ (X_2 X_2)_p x^3 \right] (x^2)^2\end{eqnarray}  Equations (\ref{x12ndline}), (\ref{x22ndline}), and (\ref{x32ndline}) follow from the equality $X_3 = X_1 X_2 - X_2 X_1$.  By the Inverse Function Theorem, $(\bar{x}^1, \bar{x}^2, \bar{x}^3)$ are coordinate functions on some open neighborhood $U' \subseteq U$ of $p$.  The reader can easily verify that parts (a)--(d) of the lemma hold for $(\bar{x}^1, \bar{x}^2, \bar{x}^3)$ on $U'$.
\end{proof}

\begin{lemma} \label{3ordlem} Suppose $\{X_1, X_2\}$ is a Carnot frame for $\mathbb{H}$ at $p$ on $U$.  There exist coordinate functions $(x^1, x^2, x^3)$ on an open neighborhood $U'' \subseteq U$ such that parts (a)--(d) of Lemma \ref{2ordlem} hold and \begin{equation} (X_i X_j X_k)_p x^l = 0 \end{equation} for all $i, j, k \in \{1, 2\}$ and $l \in \{1,2,3\}$. \end{lemma}

\begin{proof} By Lemma \ref{2ordlem}, there exists coordinate functions $(x^1, x^2, x^3)$ on some open neighborhood $U' \subseteq U$ of $p$ such that parts (a)--(d) of Lemma \ref{2ordlem}.  Define the following functions for $q \in U'$: \begin{eqnarray} \phi^{(1,1,1)}(q) & = & \frac {1} {6} (x^1(q))^3 \\
\phi^{(1,1,2)}(q) & = & \frac {1} {4} (x^1(q))^2 x^2(q) + \frac {1} {2} x^1(q) x^3 (q) \\
\phi^{(1,2,1)}(q) & = & 0 \\
\phi^{(1,2,2)}(q) & = &  \frac {1} {4} x^1(q) (x^2(q))^2 + \frac {1} {2} x^2(q) x^3 (q) \\
\phi^{(2,1,1)}(q) & = &  \frac {1} {4} (x^1(q))^2 x^2(q) - \frac {1} {2} x^1(q) x^3 (q) \\
\phi^{(2,1,2)}(q) & = & 0 \\
\phi^{(2,2,1)}(q) & = & \frac {1} {4} x^1(q) (x^2(q))^2 - \frac {1} {2} x^2(q) x^3 (q) \\
\phi^{(2,2,2)}(q) & = & \frac {1} {6} (x^2(q))^3 
\end{eqnarray} We note that for $i, j, k, \alpha, \beta, \gamma \in \{1, 2\}$ except $(i, j, k) \in \{(1,2,1), (2,1,2) \}$ we have that \begin{equation} (X_i X_j X_k)_p \phi^{(\alpha, \beta, \gamma)} = \delta_i^\alpha \delta_j^\beta \delta_k^\gamma \end{equation} and for $(i, j, k) \in \{(1,2,1), (2,1,2) \}$, we have that \begin{equation} (X_i X_j X_k)_p \phi^{(\alpha, \beta, \gamma)} = 0 \end{equation}  Now define, for $l \in \{1,2,3\}$, the following for $q \in U'$: \begin{equation} \bar{x}^l(q) = x^l(q) - \sum_{i, j, k \in \{1,2\}} \left[ (X_i X_j X_k)_p x^l \right] \phi^{(i, j, k)}(q) \end{equation} By the Inverse Function Theorem, there exists an open neighborhood $U'' \subseteq U'$ on which $(\bar{x}^1, \bar{x}^2, \bar{x}^3)$ are coordinate functions, and the reader can easily verify that $(\bar{x}^1, \bar{x}^2, \bar{x}^3)$ satisfy parts (a)--(d) of Lemma \ref{2ordlem}.  We also see quickly that for all $l \in \{1,2,3\}$ and for all $i, j, k \in \{1, 2\}$ except $(i, j, k) \in \{(1,2,1), (2,1,2) \}$, we have that \begin{equation} \label{3rdord} (X_i X_j X_k)_p x^l = 0 \end{equation} To show the same for $(i, j, k) \in \{(1,2,1), (2,1,2) \}$, we note that, because $\{X_1, X_2 \}$ is a Carnot frame for $\mathbb{H}$ at $p$, it follows that \begin{eqnarray} (X_1 X_2 X_1)_p & = & \frac {1} {2} (X_1 X_1 X_2)_p + \frac {1} {2} (X_2 X_1 X_1)_p - \frac {1} {2} [X_1, X_3]_p \\ & = & \frac {1} {2} (X_1 X_1 X_2)_p + \frac {1} {2} (X_2 X_1 X_1)_p \\ (X_2 X_1 X_2)_p & = & \frac {1} {2} (X_2 X_2 X_1)_p + \frac {1} {2} (X_1 X_2 X_2)_p + \frac {1} {2} [X_2, X_3]_p \\ & = & \frac {1} {2} (X_2 X_2 X_1)_p + \frac {1} {2} (X_1 X_2 X_2)_p
\end{eqnarray} Thus equation (\ref{3rdord}) for $(i, j, k) \not\in \{(1,2,1), (2,1,2) \}$ implies equation (\ref{3rdord}) for $(i, j, k) \in \{(1,2,1), (2,1,2) \}$.
\end{proof}

\begin{theorem} \label{carnotflatequiv} Let $\{X_1, X_2, X_3\}$ be as above.  The following are equivalent.
\begin{enumerate}
\item[(a)] $\{X_1, X_2\}$ is a Carnot frame for $\mathcal{H}$ at $p$.
\item[(b)] On some open neighborhood $\hat{U} \subseteq U$ of $p$, there exists a flattening of $g$ at $p$ with frame $\{\hat{X}_1, \hat{X}_2, \hat{X}_3 \}$ such that, for all first, second, and third weighted order multi-indices $\alpha$, \begin{equation} \hat{X}^\alpha_p = X^\alpha_p \end{equation} as smooth differential operators at p..
\end{enumerate}
\end{theorem}

\begin{proof}
$(b) \implies (a)$.  Assuming (b), we have that \begin{eqnarray} [X_1, X_3]_p & = & (X_1 X_1 X_2)_p - 2 (X_1 X_2 X_1)_p + (X_2 X_1 X_1)_p \\ & = & (\hat{X}_1 \hat{X}_1 \hat{X}_2)_p - 2 (\hat{X}_1 \hat{X}_2 \hat{X}_1)_p + (\hat{X}_2 \hat{X}_1 \hat{X}_1)_p \\ & = & [\hat{X}_1, \hat{X}_3]_p = 0  \end{eqnarray} and similarly \begin{eqnarray} [X_2, X_3]_p & = & 2 (X_2 X_1 X_2)_p - (X_2 X_2 X_1)_p - (X_1 X_2 X_2)_p \\ & = & 2 (\hat{X}_2 \hat{X}_1 \hat{X}_2)_p - (\hat{X}_2 \hat{X}_2 \hat{X}_1)_p - (\hat{X}_1 \hat{X}_2 \hat{X}_2)_p \\ & = & [\hat{X}_2, \hat{X}_3]_p = 0  \end{eqnarray}  Thus $\{X_1, X_2\}$ is a Carnot frame for $\mathcal{H}$ at $p$.

$(a) \implies (b)$.  By Lemma \ref{3ordlem}, there exists coordinate functions $(x^1, x^2, x^3)$ on some open neighborhood $U'' \subseteq U$ of $p$ such that parts (a)--(d) of Lemma \ref{2ordlem} hold and \begin{equation} (X_i X_j X_k)_p x^l = 0 \end{equation} for all $i, j, k, \in \{1, 2\}$ and $l \in \{1,2,3\}$.  Define the following smooth vector fields on $U''$: \begin{equation} \hat{X}_1 = \frac {\partial} {\partial x^1} - \frac {1} {2} x^2 \frac {\partial} {\partial x^3} \qquad \hat{X}_2 = \frac {\partial} {\partial x^2} + \frac {1} {2} x^1 \frac {\partial} {\partial x^3} \qquad \hat{X}_3 = \frac {\partial} {\partial x^3} \end{equation} A quick calculation shows that the Heisenberg bracket relations hold for the frame $\{ \hat{X}_1, \hat{X}_2, \hat{X}_3 \}$ on $U''$.  The reader can also verify that
\begin{enumerate}
\item[(i)] $(\hat{X}_i)_p  x^j = \delta^i_j$ for all $i, j \in \{1,2,3\}$;
\item[(ii)] $(\hat{X}_i \hat{X}_j)_p x^k = 0$ for all $i, j, k \in \{1, 2, 3\}$ except $(i, j, k) \in \{(1, 2, 3), (2, 1, 3)\}$;
\item[(iii)] $(\hat{X}_1 \hat{X}_2)_p x^3 = \frac {1} {2}$ and $(\hat{X}_2 \hat{X}_1)_p x^3 = -\frac {1} {2}$; and
\item[(iv)] $(\hat{X}_i \hat{X}_j \hat{X}_k)_p x^l = 0$ for all $i, j, k \in \{1, 2\}$ and $l \in \{1,2,3\}$.
\end{enumerate}
Thus we have that, for all first, second, and third weighted order multi-indices $\alpha$, \begin{equation} (\hat{X}_\alpha)_p x^i = (X_\alpha)_p x^i \end{equation} for $i \in \{1, 2, 3\}$.  Applying the Leibniz rule, we get that for all multi-indices $\alpha$ such that $|\alpha|_w \leq 3$ and for all multi-indices $\beta$, we have that \begin{equation} \label{equalphabeta} (\hat{X}_\alpha)_p x^\beta = (X_\alpha)_p x^\beta \end{equation} and in particular, if $|\beta|_w > |\alpha|_w$, then both sides of the above equation equal 0.  To complete the proof, let $\hat{U} \subseteq U$ be an open disc centered at $p$.  For all smooth function $f$ on $\hat{U}$, the third order Taylor's formula tells us that for any $q \in \hat{U}$ with coordinates $(x^1(q), x^2(q), x^3(q))$, \begin{equation} \label{taylor} f(q) = \sum_{|\beta| \leq 3} \frac {1} {|\beta| !} x^\beta(q) (\partial_\beta f)(p) + \sum_{|\beta| = 4} \frac {1} {6} x^\beta(q) \int_0^1 (1 - t)^3 (\partial_\beta f)(\gamma(t)) \, \dif t \end{equation} where $\gamma : [0,1] \rightarrow \hat{U}$ is the continuous path defined in terms of the coordinate functions $(x^1, x^2, x^3)$ by \begin{equation} \gamma(t) = (t x^1(q), t x^2(q), t x^3 (q)) \end{equation}  Using equations (\ref{equalphabeta}) and (\ref{taylor}), we can clearly see that \begin{equation} (\hat{X}_\alpha)_p f = (X_\alpha)_p f \end{equation} for all multi-indices $\alpha$ such that $|\alpha|_w \leq 3$ and for all smooth functions $f$ on $\hat{U}$, proving this direction of the theorem.
\end{proof}

Because of Theorem \ref{carnotflatequiv}, we can now fix the direction of $\hat{X}_3$ at $p$.

\begin{corollary} \label{torsionvaluecor} Let $\{X_1, X_2\}$ be any ordered orthonormal frame for $\mathcal{H}$ on $U$, let $X_3 = [X_1, X_2]$ and let $\{\xi_1, \xi_2, \xi_3\}$ be the dual frame to $\{X_1, X_2, X_3\}$ on $U$. Suppose $\hat{g}$ is a flattening of $g$ at $p$ with frame $\{\hat{X}_1, \hat{X}_2, \hat{X}_3 \}$.  Then \begin{equation} (\hat{X}_3)_p = [(\xi_3)_p ([X_2, X_3])] (X_1)_p + [(\xi_3)_p ([X_3, X_1])] (X_2)_p + (X_3)_p \end{equation} \end{corollary}

\begin{proof}
According to Theorem \ref{carnotexistence}, there exists a Carnot frame $\{\bar{X}_1, \bar{X}_2\}$ for $\mathcal{H}$ at $p$ on $U$, and equation (\ref{x3value}), along with equations (\ref{thetaderivative1}) through (\ref{thetaderivative6}), tell us that \begin{equation} \label{carnotx3} (\bar{X}_3)_p = [(\xi_3)_p ([X_2, X_3])] (X_1)_p + [(\xi_3)_p ([X_3, X_1])] (X_2)_p + (X_3)_p \end{equation}  Corollary \ref{carnottorsion} then tells us that equation (\ref{carnotx3}) holds for all Carnot frames.

Because $\hat{g}$ is a flattening of $g$ at $p$ with frame $\{\hat{X}_1, \hat{X}_2, \hat{X}_3 \}$, we know from Theorem \ref{carnotflatequiv} that there exists a Carnot frame $\{X'_1, X'_2\}$ such that $(X'_3)_p = (\hat{X}_3)_p$.  Corollary \ref{carnottorsion} and equation (\ref{carnotx3}) then imply the corollary.
\end{proof}

\section{The Natural Connection on (2,3) Sub-Riemannian Manifolds}

We now come to the main result of this paper.  First, we define the natural connection on $M$.

\begin{definition} Let $\{X_1, X_2 \}$ be an ordered orthonormal frame for $\mathcal{H}$ on $U$ as above. The \emph{natural connection} for the sub-Riemannian structure $\left( \mathcal{H}, \left< \cdot, \cdot \right> \right)$ on $U$ is the affine connection compatible with $\left( \mathcal{H}, \left< \cdot, \cdot \right> \right)$ such that, for all $q \in U$,
\begin{enumerate}
\item[(a)] $T(X_1, X_2) = -[(\xi_3)_q ([X_2, X_3])] X_1 -[(\xi_3)_q ([X_3, X_1])] X_2 - X_3$
\item[(b)] $T(X_2, X_3) = (X_2 \theta) [(\xi_3)_q ([X_2, X_3])] X_1 +(X_2 \theta) [(\xi_3)_q ([X_3, X_1])] X_2 + (X_2 \theta) \, X_3$
\item[(c)] $T(X_3, X_1) = (X_1 \theta) [(\xi_3)_q ([X_2, X_3])] X_1 + (X_1 \theta) [(\xi_3)_q ([X_3, X_1])] X_2 + (X_1 \theta) \, X_3$
\item[(d)] $R_\mathcal{H} V = 0$ for all $V \in \mathcal{X}(U)$
\end{enumerate}
\end{definition}

The following theorem, the main theorem of the paper, tells us the main property of the natural connection: that, for any point $p$ in $M$ and any flattening $\hat{g}$ of $g$ at $p$, the natural connection agrees at $p$ with the connection the flattening inherits from the Heisenberg group, as derived in Section 3.  Thus the natural connection is the unique affine connection compatible with the sub-Riemannian structure of $M$ that agrees with the parallel structure of the Heisenberg group, just as the Levi-Civita connection is the unique connection compatible with the metric of a Riemannian manifold that agrees with the parallel structure of Euclidean space.

\begin{theorem} \label{maintheoremofpaper} Let $\nabla$ be the natural connection on $M$ compatible with the sub-Riemannian structure $\left( \mathcal{H}, \left< \cdot, \cdot \right> \right)$.  Let $p \in M$, and let $\hat{g}$ be a flattening of $g$ at $p$ with frame $\{\hat{X}_1, \hat{X}_2, \hat{X}_3 \}$ on an open neighborhood $\hat{U}$ of $p$. Let $\hat{\nabla}$ be the affine connection on $\hat{U}$ with co-metric $\hat{g}$ inherited through its isometry with an open neighborhood of the Heisenberg group $\mathbb{H}^1$.  Then $\nabla = \hat{\nabla}$ at $p$.
\end{theorem}

\begin{proof} According to Lemma \ref{torsioncurvature}, at $p$, because $\hat{\nabla}$ is inherited from $\mathbb{H}^1$, the torsion and horizontal curvature operator of $\hat{\nabla}$ have values
\begin{enumerate}
\item[(i)] $\hat{T}(\hat{X}_1, \hat{X}_2) = -\hat{X}_3$
\item[(ii)] $\hat{T}(\hat{X}_2, \hat{X}_3) = 0$
\item[(iii)] $\hat{T}(\hat{X}_3, \hat{X}_1) = 0$
\item[(iv)] $\hat{R}_\mathcal{H} V = 0$ for all $V \in \mathcal{X}(\hat{U})$
\end{enumerate}
Corollary \ref{torsionvaluecor} tells us that statement (i) can be rewritten at $p$ as \begin{equation} (\hat{T}(X_1, X_2))_p = -(\hat{X}_3)_p = -[(\xi_3)_p ([X_2, X_3])] (X_1)_p -[(\xi_3)_p ([X_3, X_1])] (X_2)_p - (X_3)_p \end{equation} We then use statement (ii) to show that \begin{eqnarray} 0 & = & \hat{T}(\hat{X}_2, \hat{X}_3) \\ & = & \hat{T}( -\sin \theta \, X_1 + \cos \theta \, X_2, -(X_1 \theta) \, X_1 - (X_2 \theta) \, X_2 + X_3) \\ & = & [ (X_2 \theta) \sin \theta + (X_1 \theta) \cos \theta ] \hat{T}(X_1, X_2) + \sin \theta \, \hat{T}(X_3, X_1) + \cos \theta \, \hat{T}(X_2, X_3) \\ & = & -[ (X_2 \theta) \sin \theta + (X_1 \theta) \cos \theta ] \hat{X}_3 + \sin \theta \, \hat{T}(X_3, X_1) + \cos \theta \, \hat{T}(X_2, X_3) \\ \label{TX2X3} & = & \cos \theta \, ( \hat{T}(X_2, X_3) - (X_1 \theta) \, \hat{X}_3) + \sin \theta \, (\hat{T}(X_3, X_1) - (X_2 \theta) \hat{X}_3) \end{eqnarray} A similar calculation with statement (iii) tells us that \begin{equation} \label{TX3X1} 0 =  -\sin \theta \, ( \hat{T}(X_2, X_3) - (X_1 \theta) \, \hat{X}_3) + \cos \theta \, (\hat{T}(X_3, X_1) - (X_2 \theta) \hat{X}_3) \end{equation}  Equations (\ref{TX2X3}) and (\ref{TX3X1}) together imply that \begin{equation} \hat{T}(X_2, X_3) = (X_1 \theta) \, \hat{X_3} \qquad \text{and} \qquad \hat{T}(X_3, X_1) = (X_2 \theta) \, \hat{X_3} \end{equation} which, together with Corollary \ref{torsionvaluecor}, implies that \begin{equation} \hat{T}(X_2, X_3) = (X_2 \theta) [(\xi_3)_q ([X_2, X_3])] X_1 +(X_2 \theta) [(\xi_3)_q ([X_3, X_1])] X_2 + (X_2 \theta) \, X_3 \end{equation} and \begin{equation} \hat{T}(X_3, X_1) = (X_1 \theta) [(\xi_3)_q ([X_2, X_3])] X_1 + (X_1 \theta) [(\xi_3)_q ([X_3, X_1])] X_2 + (X_1 \theta) \, X_3 \end{equation} Thus $T = \hat{T}$ and $R_\mathcal{H} = \hat{R}_\mathcal{H}$ at $p$.  Thus, according to Theorem \ref{comptheorem}, $\nabla = \hat{\nabla}$ at $p$.   
\end{proof}

\section*{References}

\noindent [Bel] Bella\"{i}che, Andr\'{e}. \emph{The tangent space in sub-Riemannian geometry.} Sub-Riemannian geometry, 1--78, Progr. Math., 144, BirkhŠuser, Basel, 1996. \\

%\noindent [CHMY] Cheng, Jih-Hsin; Hwang, Jenn-Fang; Malchiodi, Andrea; Yang, Paul \emph{Minimal surfaces in pseudohermitian geometry.} Ann. Sc. Norm. Super. Pisa Cl. Sci. (5) 4 (2005), no. 1, 129--177. \\

%\noindent [Cho] Chow, Wei-Liang. \emph{†ber Systeme von linearen partiellen Differentialgleichungen erster Ordnung.} (German) Math. Ann. 117, (1939). 98--105. \\

\noindent [HP] Robert K. Hladky and Scott D. Pauls. \emph{Constant mean curvature surfaces in sub-Riemannian geometry.} http://arXiv:math/0508333v1 (2005). \\

\noindent [HP2] Robert K. Hladky and Scott D. Pauls. \emph{Minimal surfaces in the roto-translation group with applications to a neuro-biological image completion model.} http://arXiv:math/0509636v1 (2005). \\

%\noindent [Mit] Mitchell, John. \emph{On Carnot-CarathŽodory metrics.} J. Differential Geom. 21 (1985), no. 1, 35--45. \\

\noindent [Mon] Montgomery, Richard. \emph{A tour of subriemannian geometries, their geodesics and applications.} Mathematical Surveys and Monographs, 91. American Mathematical Society, Providence, RI, 2002. xx+259 pp. ISBN: 0-8218-1391-9 \\

%\noindent [Pau] Pauls, Scott D. \emph{Minimal surfaces in the Heisenberg group.} Geom. Dedicata 104 (2004), 201--231. \\

%\noindent [Spi] Spivak, Michael. \emph{A comprehensive introduction to differential geometry.} Vol. II. Second edition. Publish or Perish, Inc., Wilmington, Del., 1979. xv+423 pp. ISBN: 0-914098-83-7 \\

\noindent [Tan] Tanaka, Noboru. \emph{A differential geometric study on strongly pseudo-convex manifolds.} Lectures in Mathematics, Department of Mathematics, Kyoto University, No. 9. Kinokuniya Book-Store Co., Ltd., Tokyo, 1975. iv+158 pp. \\

\noindent [Web] Webster, S.\@ M. \emph{Pseudo-Hermitian structures on a real hypersurface.} J. Differential Geom. 13 (1978), no. 1, 25--41.

\end{document}